\documentclass{amsart}

\usepackage[OT2, T1]{fontenc}
\usepackage{url}
\usepackage{amsmath}
\usepackage{array}
\usepackage{graphicx}
\usepackage{amsfonts}
\usepackage{amssymb}
\usepackage{amstext}
\usepackage{amsthm}
\usepackage{hyperref}
\usepackage{colonequals}
\usepackage{enumitem}
\usepackage[alphabetic,lite]{amsrefs}
\usepackage{cleveref}
\usepackage[all,cmtip]{xy}
\usepackage{fullpage}
\usepackage{mathtools}
\usepackage{tikz-cd}

\numberwithin{equation}{subsection}

\newtheorem{theorem}{Theorem}
\numberwithin{theorem}{section}
\newtheorem{lemma}[theorem]{Lemma}

\theoremstyle{definition}
\newtheorem{defn}[theorem]{Definition}

\theoremstyle{remark}
\newtheorem{rem}[theorem]{Remark}

\newtheorem{example}[theorem]{Example}

\newcommand{\bC}{\mathbb{C}}

\newcommand{\bF}{\mathbb{F}}

\newcommand{\bN}{\mathbb{N}}
\newcommand{\bP}{\mathbb{P}}
\newcommand{\bQ}{\mathbb{Q}}
\newcommand{\bR}{\mathbb{R}}

\newcommand{\bZ}{\mathbb{Z}}

\newcommand{\cA}{\mathcal{A}}

\newcommand{\cC}{\mathcal{C}}

\newcommand{\cE}{\mathcal{E}}

\newcommand{\cH}{\mathcal{H}}

\newcommand{\cO}{\mathcal{O}}
\newcommand{\cP}{\mathcal{P}}

\newcommand{\cV}{\mathcal{V}}

\newcommand{\cX}{\mathcal{X}}

\newcommand{\fp}{\mathfrak{p}}

\newcommand{\ul}[1]{\underline{\smash{#1}}}

\DeclareMathOperator{\PGL}{PGL}
\DeclareMathOperator{\PSL}{PSL}

\DeclareMathOperator{\Nm}{Nm}
\DeclareMathOperator{\trd}{trd}
\DeclareMathOperator{\nrd}{nrd}

\DeclareMathOperator{\End}{End}
\DeclareMathOperator{\Hom}{Hom}

\DeclareMathOperator{\Spec}{Spec}

\DeclareMathOperator{\sgn}{sgn}

\newcommand{\twocase}[5]{#1 \begin{cases} #2 & \text{{\rm #3}}\\ #4 &\text{{\rm #5}} \end{cases}}

\begin{document}
    \title[]{Superspecial primes for QM abelian surfaces over real number fields}

\author {Fangu Chen\\
\texttt{\lowercase{fangu@berkeley.edu}}}

\begin{abstract}
    Baba and Granath generalize Elkies' theorem on infinitude of supersingular primes for elliptic curves to abelian surfaces with quaternionic multiplication of discriminant $6$, whose field of moduli is $\bQ$ and which is a Jacobian in characteristic $2$ and $3$. We extend the field of moduli to any number field with a real embedding, and weaken the local conditions at $2$ and $3$. The proof relies on the intersection theory of Heegner divisors on Shimura curves. 
\end{abstract}

\maketitle

\section{Introduction} 

\subsection{Background}

Elkies proved in \cite{MR903384} that an elliptic curve defined over a number field of odd degree has infinitely many primes of supersingular reduction, and extended the result in \cite{MR1030140} to elliptic curves over any number field with at least one real embedding. This has been generalized to various families of abelian varieties such as abelian surfaces with multiplication by the quaternion algebra of discriminant $6$, with field of moduli $\bQ$ and which is a Jacobian in characteristic $2$ and $3$ in \cite{MR2414789}. For a detailed discussion of related works, see \cite{chen2025infinitelysupersingularprimesmumfords}.

\subsection{The main result}
We extend the result of \cite{MR2414789} to number fields with at least one real embedding, and weaken the local conditions at the bad reduction primes $2$ and $3$. We expect that similar methods can be applied to the settings in \cites{MR2704678, MR2705896, li2025infinitelyprimesbasicreduction,chen2025infinitelysupersingularprimesmumfords}.

A special case of our main theorem is as follows. 
\begin{theorem}\label{thm_specialcase}
    Let $C$ be a genus $2$ curve with Jacobian that has multiplication by the maximal quaternion order with discriminant $6$, 
    and has field of moduli a number field $L$ with at least one real embedding. 
    Assume $C$ has potentially smooth stable reduction at primes above $2$ and $3$. Then its Jacobian has superspecial reduction at infinitely many primes. 
\end{theorem}

More generally, the coarse moduli variety of principally polarized abelian surfaces with potential multiplication by the maximal quaternion order of discriminant $6$ is isomorphic to $\bP^1$, and an arithmetic $j$-function (see \eqref{eq:j} in \cref{sec:coord}) is defined in \cite{MR2423455}. In terms of this coordinate, the assumption at primes $\fp$ above $2$ and $3$ in \Cref{thm_specialcase} corresponds to the case $v_{\fp}(j(C)) = 0$ by \cite{MR2414789}*{Proposition 2}, and our main theorem can be stated as follows. 

\begin{theorem}\label{thm_main}
Let $C$ be a genus $2$ curve with Jacobian that has multiplication by the maximal quaternion order with discriminant $6$, 
    and has field of moduli a number field $L$ with at least one real embedding. 
    Write $j_0:=j(C)\in L$ and $\Nm_{L/\bQ}(j_0) = \frac{n}{d}$ with $n,d\in \bZ, (n,d,6) = 1, d>0$. 
    Assume at least one of the following conditions:
    \begin{enumerate}
        \item $v_{\mathfrak{p}}(j_0) \leq 0$ for $\mathfrak{p}$ above $2,3$, and $[L:\bQ]+v_3(d\Nm_{L/\bQ}(27j_0+16))$ is odd or $j_0$ has at least one real conjugate in $(-\frac{16}{27},0)$;
        \item $v_{\mathfrak{p}}(j_0) \geq 0$ for $\mathfrak{p}$ above $2$, $v_{\mathfrak{p}}(j_0) \leq 0$ for $\mathfrak{p}$ above $3$, and $j_0$ has at least one real conjugate in $(0, \infty)$;
        \item $[L:\bQ]$ is even, $v_{\mathfrak{p}}(j_0) \geq 0$ for $\mathfrak{p}$ above $2,3$, $j_0$ has at least one real conjugate in $(-\infty, -\frac{16}{27})\cup (0, \infty)$.
    \end{enumerate} 
    Then the Jacobian of $C$ has superspecial reduction at infinitely many primes. 
\end{theorem}

An abelian surface with action of a maximal order in a rational quaternion algebra $B$ has either ordinary or superspecial reduction modulo primes not dividing the discriminant of $B$ (\cite{Clark2003}*{p.70} and \cite{MR1097624}*{p.23}). It therefore suffices to construct primes of supersingular reduction, for which we follow Elkies' general strategy. Given a genus $2$ curve $C$, we find supersingular reduction of its Jacobian from its intersection with some Heegner cycle $\cP_D$ of discriminant $D$. The intersection at a prime $p$ is captured by $v_p(P_D(j_0))>0$, where $P_D(x)$ is the integral minimal polynomial of the $j$-invariants of the points in $\cP_D$, and when this occurs, supersingular reduction at $p$ is detected by $\left(\frac{D}{p}\right) \neq 1$. 

As in \cite{MR2414789}, we need to consider certain elliptic point in addition to the Heegner cycle in order to pair the roots of $P_D(x)$ modulo primes dividing $D$. We use \cite{MR2704678} to generalize \cite{MR2414789} through a more detailed, case-by-case study, where in each case the discriminant $D$ is chosen in a form adapted to the conditions. 
In characteristic $2$ and $3$, \cite{MR2414789} shows that the reduction of any Heegner cycle lies in the superspecial locus $\{j = 0, \infty\}$, and the intersection formula of Heegner divisors in \cite{MR2441697} provides the new input that determines the specific superspecial point for each chosen Heegner cycle. The local conditions on $j_0$ at primes above  $2$ and $3$ ensure that it avoids intersection with the chosen Heegner cycles at these primes. 
For more on the choice of local conditions and Heegner cycles, and for potential further weakenings of these conditions, see \Cref{rem:conclude}.

\subsection{Notation and conventions}

Assume the following unless specified otherwise.

Let $B = B_{\Delta}$ be an indefinite quaternion algebra over $\bQ$ of discriminant $\Delta$,  $\Lambda = \Lambda_{\Delta}$ be a maximal order of $B$, and $\Lambda^1 = \Lambda_{\Delta}^1$ be the group of units in $\Lambda$ of norm $1$. Fix an element $\mu \in \Lambda_{\Delta}$ such that $\mu^2 = -\Delta$,\footnote{The element $\mu$ exists because the field $\bQ(\sqrt{-\Delta})$ embeds in $B$ by the local-global principle, and any two maximal orders in $B$ are conjugate to each other by strong approximation.} then the involution $\alpha \mapsto \alpha' = \mu^{-1} \bar{\alpha} \mu$ is a positive anti-involution on $B$.

\section{Preliminaries}

In this section, we recall the setup in \cite{MR2414789}. 
\subsection{The Shimura modular curves}

    Let $V_{\Delta}$ be the Shimura curve over $\bQ$ associated to $\Lambda$. It is the coarse moduli space of isomorphism classes of $[A, \iota]$, where $A$ is a principally polarized abelian surface, and $\iota: \Lambda  \xhookrightarrow{} \End(A)$ is an embedding such that the Rosati involution defined by the polarization on $\iota(\Lambda_{\Delta})$ is $'$.\footnote{Given an abelian surface $A$ and $\iota: \Lambda \xhookrightarrow{} \End(A)$, there is a unique principal polarization on $A$ such that the induced Rosati involution is compatible with $\mu$. }  
    Fix an isomorphism  $\iota_{\infty}: B_{\bR} \xrightarrow{\sim} M_2(\bR)$. For any point $\tau$ on the upper half plane $\cH$, consider the complex torus $A_{\tau} = \bC^2 / \Lambda_{\tau}$, where \[\Lambda_{\tau} = \iota_{\infty}(\Lambda)\begin{pmatrix}
        \tau \\ 1
    \end{pmatrix}\] is a lattice in $\bC^2$, the map $\iota_{\infty}$ induces a natural embedding $\iota_{\tau}: \Lambda_{\tau} \xhookrightarrow{} \End(A_{\tau})$ as $\Lambda$ is closed under multiplication. After possibly replacing $\mu$ by $-\mu$ once and for all, the Riemann form \begin{align*}
        E_{\tau}: \Lambda_{\tau}\times\Lambda_{\tau} &\to \bZ, \\ \left(\iota_{\infty}(x)\begin{pmatrix}
        \tau \\ 1
    \end{pmatrix}, \iota_{\infty}(y)\begin{pmatrix}
        \tau \\ 1
    \end{pmatrix}\right) &\mapsto \frac{1}{\Delta}\trd(\mu x \bar{y})
    \end{align*} defines the unique principal polarization on $A_{\tau}$ compatible with $\iota_{\tau}$.    
    Let $\Gamma^1$ be the image in $\PSL_2(\bR)$ of $\iota_{\infty}(\Lambda^1)$, then the map $\Gamma^1 \tau \mapsto [A_{\tau}, \iota_{\tau}]$ gives a bijection $\Gamma^1\backslash\cH \to V_{\Delta}(\bC)$. 

    There is the natural forgetful map $V_{\Delta} \to \cA_2$ from $V_{\Delta}$ to the moduli threefold of principally polarized abelian surfaces under which $[A, \iota]$ maps to $[A, \cC]$ where $\cC$ is the unique principal polarization on $A$ compatible with $\iota$. Let $E_{\Delta}$ denote its image and $E_{\Delta}^0$ denote the intersection of $E_{\Delta}$ with the image of the moduli space of genus $2$ curves under the open Torelli map. 
    
    From now on, let $\Delta = 6$. By \cite{MR2034317}, 
    the forgetful map $V_{6} \to \cA_2$ has degree $4$ and factors as $V_{6} \to V_{6}/W_{6} \simeq E_{6} \xhookrightarrow{} \cA_2$, where the Atkin-Lehner group $W_6 = N_{B^\times}(\Lambda)/\bQ^{\times}\Lambda^\times =\{1, w_2, w_3, w_6\} \simeq \prod_{p|6} \bZ/2\bZ$, and $w_d$ is represented by some $\chi_d \in \cO \cap N_{B^\times}(\Lambda)$ with $\trd(\chi_d) = 0, \, \nrd(\chi_d) = d$. 
    Moreover, there is a unique $N_{B^\times}(\Lambda)$-conjugacy class of embedding $\bZ[\sqrt{-6}] \xhookrightarrow{}\Lambda$, then the image $E_6$ is independent of the choice of $\mu$ and $E_6$ is the moduli space of principally polarized abelian surfaces with potential $QM$ by $\Lambda$. 

    \[\begin{tikzcd}
	& {V_6} \\
	{E_6^0} & {E_6} \\
	{\mathcal{M}_2} & {\mathcal{A}_2}
	\arrow["{W_6}", from=1-2, to=2-2]
	\arrow[hook, from=2-1, to=2-2]
	\arrow[hook, from=2-1, to=3-1]
	\arrow[hook, from=2-2, to=3-2]
	\arrow[hook, from=3-1, to=3-2]
\end{tikzcd}\]
\subsection{CM points}
    Let $K$ be an imaginary quadratic field and suppose $K\xhookrightarrow{} B$ embeds. Let $\cO_D = K\cap \Lambda$ be an order of discriminant $D$, then the fixed point $\tau$ of $\iota_{\infty} (\cO_D)$ in $\cH$ is a CM point, its corresponding abelian variety $A_{\tau}$ has $\End_{\mathrm{QM}}(\cA_{\tau}) \simeq \cO_D$. 
    \begin{rem}
        By Eichler, the number of $\Lambda^\times$-conjugacy classes of optimal embeddings $\cO_D\xhookrightarrow{} \Lambda$ is \[s(\cO_D):=h(\cO_D) \prod_{p|\Delta}\left(1-\left(\frac{\cO_D}{p}\right)\right), \] where 
        \[\twocase{\left(\frac{\cO_D}{p}\right):=}{1}{if $p$ divides the conductor of $\cO_D$,}{\left(\frac{K}{p}\right)}{otherwise.}\]Note that $s(\cO_D) \neq 0$ if and only if the conductor of $\cO_D$ is relatively prime to $\Delta$ and $K$ splits $B$. Therefore, in the case $2|\Delta$, there are no points with CM by an imaginary quadratic order of conductor $2$.  
    \end{rem}
    From now on, let $D$ be a fundamental discriminant such that $\cO_D\xhookrightarrow{} \Lambda$ and $E_6(D)$ denote the set of points with CM by $\cO_D$. The Hilbert class group of $K$ acts on $E_6(D)$ and the complex conjugation preserves $E_6(D)$. Let $W''$ denote the subgroup of $W$ generated by elements $w_p$, where $p|\Delta$ is a prime ramified in $K$, then the number of elements in $E_6(D)$, counted with appropriate multiplicities, is \[h' = \frac{h(\cO_D)}{\#W''}. \] By genus theory and \cite{MR963648}*{19.6}, the parity of $h'$ can be determined. Let $l$ be a prime. The relevant cases considered in this work are summarized in the following table.
    
    \begin{table}[h]
\centering
\begin{tabular}{|c|c|c|}
\hline
$D$ &  $h(\cO_D)$ & $h'$ \\
\hline
$-4l, \, l\equiv 13\bmod 24$          &   $\equiv 2\pmod 4$          &    odd      \\
$-l, \ l\equiv 19\bmod 24$    &      odd       &     odd        \\
$-3l, \, l\equiv 1\bmod 24$    &  $\equiv 0\pmod 4$         &    even       \\
\hline
\end{tabular}
\end{table}
    
\subsection{The coordinate functions} \label{sec:coord}
    It is shown in \cite{MR2423455}*{3.6} that there is an isomorphism $j = j_6 : E_6^0 \to \bP^1\backslash\{0, \infty\}$ given by \begin{equation} \label{eq:j}
        j = \frac{12^{10} J_{10}^2}{(J_2^2 - 24 J_4)^5}.
    \end{equation}
Denote the elements of $E_6(D)$ to be $a_1, \dotsc, a_{h'}$, and define \[P_D(x) = b_{h'} \prod_{i=1}^{h'} (x - j(a_i)),\] where $b_{h'} > 0$ is the smallest integer such that $P_D(x) \in \bZ[x]$. 
    
Elkies (\cite{MR1726059}) computed a different rational coordinate function $t: E_6\to \bP^1$, which is used in \cite{MR2704678}. The relation between the two coordinate functions is $j = 16(t-1)/27$. In particular, from the following correspondence of the coordinates of the elliptic points we can translate the results in \cite{MR2704678} to our setting. 

\begin{table}[h]
\centering
\begin{tabular}{|c|c|c|c|}
\hline
\text{Elliptic point of order} & \text{CM} & $j$ & $t$ \\
\hline
6 & $\mathbb{Z}\left[\frac{1+\sqrt{-3}}{2}\right]$ & $\infty$ & $\infty$ \\
4 & $\mathbb{Z}\left[\sqrt{-1}\right]$ & $0$ & $1$ \\
2 & $\mathbb{Z}\left[\sqrt{-6}\right]$ & $-\dfrac{16}{27}$ & $0$ \\
\hline
\end{tabular}
\end{table}

\begin{lemma}[\cite{MR2704678}*{3.3.1}]
    Let $l \geq 5$ be a prime. 
    For discriminants $D$ of the form in the table, the polynomial $P_D(x)$ has at most one real root in each of the intervals $I_1 = (-\infty, -\frac{16}{27})$, $I_2 = ( -\frac{16}{27}, 0)$, and $I_3 = (0, \infty)$, with the roots being located as follows: (see \cite{MR2704678}*{3.3.1} for a complete list, $*$ denotes the presence of a real root in the interval)

\begin{center}
\begin{tabular}{|c|c|c|c|}
\hline
$D$ &  $I_1$ & $I_2$ & $I_3$ \\
\hline
$-4l, \, l\equiv 13\bmod 24$          &             &    $*$     &         \\
$-l, \ l\equiv 19\bmod 24$    &             &      & $*$         \\
$-3l, \, l\equiv 1\bmod 24$    &  $*$         &     &     $*$        \\
\hline
\end{tabular}\end{center}

Furthermore, for any subinterval and sub-congruence class of any starred entry, there exist infinitely many primes $l$ in that sub-congruence class for which $P_D(x)$ has a real root in the subinterval. 
\end{lemma}

\begin{rem}
    Let $F_1 = \bQ(\sqrt{2})$, $F_2 = \bQ(\sqrt{3})$, $F_3 = \bQ(\sqrt{6})$, and $\varepsilon_i$ be a fundamental unit of $F_i$, $i=1,2,3$. 
    The claim that any subinterval of an interval $I_i$, in which polynomials $P_D$ with $D$ of the specified form have real roots, supports infinitely many such $D$ follows from an equidistribution result (\cite{MR1282723}*{\MakeUppercase{\romannumeral 15}, \S5}) for primes in $F_i$, by constructing a Hecke character $\sigma_i: \mathbb{A}_{F_i}^{\times} \to \mathbb{R} / (\ln{\varepsilon_i}) \mathbb{Z}$ given by \begin{align*}\mathbb{A}_{F_i}^{\times}/(F_i^{\times}  \prod_{\mathfrak{p}\nmid \infty} U_{\mathfrak{p}}) \simeq (\bR^{\times} \times \bR^{\times})/U_{F_i} &\to \mathbb{R} / (\ln{\varepsilon_i}) \mathbb{Z},\\
    (a,b) &\mapsto \frac{1}{2}\ln{\frac{|a|}{|b|}}.\end{align*}
    If the polynomials $P_D$ have real roots in both $I_{i_1}$ and $I_{i_2}$, then the result generalizes to any open subset of $I_{i_1}\times I_{i_2}$ by equidistribution of primes in the compositum $F_{i_1}F_{i_2}$. For example,  
    let $F = F_1F_3$ and define $\sigma = (\sigma_1\circ \Nm_{F/F_1}, \sigma_3\circ \Nm_{F/F_3}): \mathbb{A}^{\times}_{F} \to \mathbb{R}/(\ln{\varepsilon_1})\bZ \times \mathbb{R}/(\ln{\varepsilon_3})\bZ $, then $\sigma(a,b,c,d) = (\frac{1}{2}\ln{\frac{|ac|}{|bd|}}, \frac{1}{2}\ln{\frac{|ab|}{|cd|}})$ for $(a,b,c,d) \in \bR^{\times}\times \bR^{\times}\times \bR^{\times}\times \bR^{\times} \xhookrightarrow{} \mathbb{A}_F^{\times}$. In particular, $\sigma(a,\frac{1}{a},1,1) = (\ln{|a|}, 0)$ and $\sigma(a,1,\frac{1}{a},1) = (0,\ln{|a|})$, so the restriction of $\sigma$ to the subgroup of ideles of norm $1$ is surjective. Let $\tau:\{\text{primes of } F\}\to \mathbb{A}_F^{\times}$ be the map taking a prime $\mathfrak{p}$ to an idele that is $1$ at all places except a prime element of $F_{\mathfrak{p}}$ at $\mathfrak{p}$, and $\lambda = \sigma\circ\tau$, then the set of primes of $F$ is $\lambda$-equidistributed in $\mathbb{R}/(\ln{\varepsilon_1})\bZ \times \mathbb{R}/(\ln{\varepsilon_3})\bZ $. Since a density $1$ of primes of $F$ lies above a totally split rational prime, the set \[\left\{\left(\frac{1}{2}\ln\left|\frac{\pi_1}{\pi_1'}\right|, \frac{1}{2}\ln\left|\frac{\pi_3}{\pi_3'}\right|\right) \, : \, l \text{ rational prime}, (l) = (\pi_1)(\pi_1') \text{ in } F_1, (l) = (\pi_3)(\pi_3')\text{ in } F_3\right\}\] is equidistributed in $\mathbb{R}/(\ln{\varepsilon_1})\bZ \times \mathbb{R}/(\ln{\varepsilon_3})\bZ $. 
\end{rem}

\begin{lemma}[\cite{MR2704678}*{3.4.1, 3.4.2}]\label{lem:pair}
    Let $l\geq 5$ be a prime. 
    For discriminants $D$ of the form in the table, each root of $P_D(x)\bmod l$ occurs with even multiplicity, except possibly for roots corresponding to points on $E_6$ which are congruent modulo $l$ to one of the three elliptic points. The divisor of unpaired zeros of $P_D(x)$ modulo $l$ is as follows: (see \cite{MR2704678}*{3.4.2} for a complete list)\begin{center}
\begin{tabular}{|c|c|}
\hline
$D$ &  divisor \\
\hline
$-4l, \, l\equiv 13\bmod 24$       &   $(-16/27)$   \\
$-l, \ l\equiv 19\bmod 24$    &   $(-16/27) $       \\
$-3l, \, l\equiv 1\bmod 24$    &  $\emptyset$    \\
\hline
\end{tabular}
\end{center}
\end{lemma}

\section{Integral model and intersection number at bad primes}
In this section, we compute $P_D(x)$ modulo $2$ and $3$ by computing intersection numbers of Heegner divisors on the integral model of $E_6$. 

Let $\mathcal{X} = \mathcal{X}_{1,6,m}$ be the moduli space for abelian surfaces with additional structure as in \cite{MR2441697}*{\S2}.  
An integral model $\mathcal{E}$ of $E_6$ (resp. $\mathcal{V}$ of $V_6$) is obtained by the quotient of $\mathcal{X}$ by a finite subgroup of order $4\cdot 2\eta(m)$ (resp. $2\eta(m)$), where $\eta(m)$ is defined in \cite{MR2441697}*{(3.4)}. 

Let $p$ be a prime. For $p\neq 2,3$, the curve $E_6$ has good reduction at $p$, and $\cE_{\bF_p} \simeq \bP^1_{\bF_p}$. Suppose $p|\Delta$. 
As a consequence of the theorem of \v{C}erednik-Drinfeld (see for instance \cite{MR1141456}), the formal completion of the integral model along its special fiber at $p$ is isomorphic to a finite union of Galois twists (over unramified extensions) of quotients of Mumford curves (\cite{MR352105}). Then the geometric special fiber  $\mathcal{E}_{\overline{\bF}_p}$ of $\mathcal{E}$ at $p$ can be viewed as a graph, and we follow \cite{MR523989} to compute its dual graph and thus determine $\mathcal{E}_{\bF_p}$.

\begin{lemma}
    For $p = 2,3$, the special fiber $\cE_{\bF_p} \simeq \bP^1_{\bF_p}$. 
\end{lemma}
\begin{proof}
    Let $B_{\Delta/p}$ be the definite quaternion algebra of discriminant $\frac{\Delta}{p}$, and $\mathfrak{O}$ be a maximal order of $B_{\Delta/p}$, so that $\mathfrak{O} \otimes \bZ_{q} \simeq \Lambda\otimes \bZ_{q}$ for any prime $q\neq p$. Let \begin{align*}
        \Gamma_0 &= \mathfrak{O}[p^{-1}]^\times / \mathbb{Z}[p^{-1}]^\times, \\
        \Gamma^* & = \{\gamma \in B_{\Delta/p}^\times: \gamma \mathfrak{O}[p^{-1}] = \mathfrak{O}[p^{-1}] \gamma\}
    \end{align*} regarded as discrete subgroups of $\PGL_2(\bQ_p)$, and $I$ denote the Bruhat-Tits tree. We have  graphs with lengths $\Gamma_0\backslash I$ and $\Gamma^*\backslash I$ defined by the quotient of the Bruhat-Tits tree \cite{MR523989}*{\S3}. The number of vertices of $\Gamma_0\backslash I$ equals the class number $h(B_{\Delta/p})$ of $B_{\Delta/p}$ \cite{MR523989}*{p.291}. 

    As explained in detail in \cite{MR523989}, the dual graph of the geometric special fiber $\mathcal{E}_{\overline{\bF}_p}$ is obtained from $\Gamma^* \backslash I$ by removing self-inverse edges. The graph $\Gamma^* \backslash I$ is a quotient of $\Gamma_0 \backslash I$, and we have $h(B_2) = 1$ and $h(B_3) =1$, so both $\mathcal{E}_{\overline{\bF}_2}$ and $\mathcal{E}_{\overline{\bF}_3}$ consist of a single irreducible component.   
    Since the generic fiber $E_6\simeq \bP^1_{\bQ}$, which has genus $0$ and a rational point, it follows that $\mathcal{E}_{\bF_2}\simeq \bP^1_{\bF_2}$ and $\mathcal{E}_{\bF_3} \simeq \bP^1_{\bF_3}$. 
\end{proof}

\subsection{Intersection number}

\begin{defn}
    Let $R$ be a Dedekind domain and $\mathcal{X} \to \Spec R$ be an arithmetic surface. 
    Let $D$ and $E$ be two effective divisors on $\cX$ with no common irreducible component. Let $z_0\in \mathcal{X}$ be a closed point. The local intersection number $i_{z_0}(D,E)$ of $D$ and $E$ at $z_0$ is the length of the $\cO_{\cX,z_0}$-module $\cO_{\cX, z_0} / (\cO_{\cX}(-D)_{z_0} + \cO_{\cX}(-E)_{z_0})$. 
\end{defn}

\begin{example}
    Let $R$ be a discrete valuation ring with field of fraction $K$, maximal ideal $\fp$ and residue field $k$. Let $x,y\in \cX(K)$ be distinct, then $x,y$ extend uniquely to $\ul{x}, \ul{y} \in \cX(R)$ by properness of $\mathcal{X} \to \Spec R$, and $\ul{x}, \ul{y} $ are closed immersions since they are sections to a separated map. Define \[(\ul{x}. \ul{y}) :=  \sum_{z\in \cX_k} i_z(\ul{x}, \ul{y}).\] Let $x_n, y_n \in \cX(R/\fp^n)$ be the reduction of $\ul{x}, \ul{y}$ modulo $\fp^n$ for positive integer $n$. Suppose $x\neq y$, 
    then as in \cite{MR2705896}*{3.13}, \[(\ul{x}. \ul{y}) = \max\{n : x_n = y_n \}. \] In particular, 
    \begin{enumerate}
        \item if $\cX$ is a fine moduli space with a universal object $\cA\to \cX$, then the complete local ring $\widehat{\cO}_{\cX, z_0}$ is the universal deformation ring for $z_0$ and \[(\ul{x}. \ul{y}) = \max\{n : \cA_{\ul{x}} \simeq \cA_{\ul{y}} \mod{\fp^n} \};\]
        \item if $\cX = \bP^1_{R}$ and $x_1 = y_1 = z_0 \in \cX(k)$, then \[(\ul{x}. \ul{y}) = i_{z_0}(\ul{x}, \ul{y}) = \begin{cases}
          v_{\fp}(x-y)   & v_{\fp}(x) \geq 0 , v_{\fp}(y)\geq0,\\
          v_{\fp}(\frac{1}{x} - \frac{1}{y}) & v_{\fp}(x) <0 \text{ or }x = \infty, v_{\fp}(y) <0 \text{ or }y = \infty.
        \end{cases}\]
    \end{enumerate}
\end{example}

Consider $\mathcal{X} \xrightarrow{2\eta(m)} \cV \xrightarrow{4} \cE$. We will use \cite{MR2441697}*{(3.12)} to compute the arithmetic intersection number. Let $D$ be a fundamental discriminant, and $\cP_{D,\cX}$ be the Heegner divisor of discriminant $D$ on $\cX$ as defined in \cite{MR2441697}*{\S2} and $\cP_{D, \cE} = \sum_{x\in E_6(D)} \ul{x}$ be the divisor on $\cE$. On $V_6$, the universal automorphism group is $\{\pm1\}$, the elliptic points of order $3$ have CM by $\bZ[\frac{1+\sqrt{-3}}{2}]$ and the elliptic points of order $2$ have CM by $\bZ[-1]$. The Atkin-Lehner involutions $w_2$ has two fixed points (CM by $\mathbb{Z}[-1]$), $w_3$ has two fixed points (CM by $\mathbb{Z}[\frac{1+\sqrt{-3}}{2}]$), $w_6$ has two fixed points (CM by $\mathbb{Z}[\sqrt{-6}]$). For discriminants $D < -6$, we have \begin{align}\label{eq:proj1}
    \langle \cP_{-3, \cE}, \cP_{D, \cE} \rangle_{\cE} =  \frac{2}{4}\cdot \frac{2\cdot 3}{2\eta(m)}\langle \mathcal{P}_{-3, \cX}, \mathcal{P}_{D, \cX}\rangle_{\mathcal{X}}, \\\label{eq:proj2}
    \langle \cP_{-4, \cE}, \cP_{D, \cE} \rangle_{\cE} =  \frac{2}{4}\cdot \frac{2\cdot 2}{2\eta(m)}\langle \mathcal{P}_{-4, \cX}, \mathcal{P}_{D, \cX}\rangle_{\mathcal{X}}. 
\end{align}
Recall \cite{MR2414789}*{Lemma 3} that $P_D(x) \equiv x^n\bmod 2$ for some $n$ and $P_D(x) \equiv \pm x^m \bmod 3$ for some $m$. For discriminants $D$ considered in this work, we obtain a strengthened version by computing the local intersection number of $\mathcal{P}_{-3, \cX}$ or  $\mathcal{P}_{-4, \cX}$ with $\mathcal{P}_{D, \cX}$ at $2$ and $3$. 

\begin{lemma}\label{lem:avoidintersection}
    Let $D$ be a fundamental discriminant. \begin{enumerate}
        \item When $3 \nmid D$, the local intersection of $\mathcal{P}_{-3,\cE}$ and $\mathcal{P}_{D, \cE}$ at $p$ is $0$ if $-3D$ is not a square modulo $24p$. 
        \item When $2\nmid D$, the local intersection of $\mathcal{P}_{-4,\cE}$ and $\mathcal{P}_{D, \cE}$ at $p$ is $0$ if $-4D$ is not a square modulo $24p$.
    \end{enumerate}
\end{lemma}
\begin{proof}
    This follows immediately from the intersection formula in \cite{MR2441697}*{Theorem 3.2} and \eqref{eq:proj1},\eqref{eq:proj2}.
\end{proof}

\begin{lemma}\label{lem:red23}
    Let $l \geq 5$ be a prime. For discriminants $D$ of the form in the table, the following reductions hold:
    \begin{table}[h]
\centering
\begin{tabular}{|c|c|c|}
\hline
$D$ &  $P_D \bmod 2$ & $P_D \bmod 3$  \\
\hline
$-4l, \, l\equiv 13\bmod 24$          &    $ x^{h'} $       &       $\pm x^{h'}$   \\
$-l, \ l\equiv 19\bmod 24$    &       $1$      &   $\pm x^{h'}$          \\
$-3l, \, l\equiv 1\bmod 24$    &  $1$         &   $\pm1$    \\
\hline
\end{tabular}
\end{table}
\end{lemma}
\begin{proof} 
    For $D=-4l$ where $l\equiv 13\bmod 24$ or $D = -l$ where $l\equiv 19 \bmod 24$,
    since $-3D$ is not a square modulo $72$, 
    the local intersection of $\mathcal{P}_{-3, \cE}$ and $\mathcal{P}_{D,\cE}$ at $p=3$ is $0$, then $v_{3}(j(a)) > 0$ for every root $j(a)$ of $P_{D}(x)$, so $P_{D}(x) \equiv \pm x^{h'}\pmod 3$. For $l\equiv 1\bmod 24$, since $12l$ is not a square modulo $72$, 
    the local intersection of $\mathcal{P}_{-4, \cE}$ and $\mathcal{P}_{-3l, \cE}$ at $p=3$ is $0$, then $v_{3}(j(a)) < 0$ for every root $j(a)$ of $P_{-3l}(x)$, so $P_{-3l}(x) \equiv \pm 1\pmod 3$.

    For $D = -l$ where $l\equiv 19\bmod 24$ or $D = -3l$ where $l\equiv 1\bmod 24$, since $-4D$ is not a square modulo $48$, the local intersection of $\mathcal{P}_{-4, \cE}$ and $\mathcal{P}_{D, \cE}$ at $p=2$ is $0$, then $v_{2}(j(a)) < 0$ for every root $j(a)$ of $P_{D}(x)$, so $P_{D}(x) \equiv 1\pmod 2$. The case of $P_{-4l} \bmod 2$ where $l\equiv 13\bmod 24$ is proved in \cite{MR2414789}*{Lemma 6}.
\end{proof}

\section{Proof of main theorem}

Let $S$ be a finite set of primes containing primes above $2,3$ and all primes occuring in $j_0$ and $27j_0 + 16$. 
    Write $P = |\Nm_{L/\bQ}(P_{D}(j_0))|, Q = |\Nm_{L/\bQ}(27j_0+16)|, s = \sgn (\Nm_{L/\bQ}(P_{D}(j_0))), s' = \sgn (\Nm_{L/\bQ}(27j_0+16))$, and $N = d^{h'}P$ where $h' = \deg P_D$. 
    Let $j_1<j_2 < \cdots < j_r$ be the real conjugates of $j_0$. 

    In each case, we are going to find some discriminant $D$ with a rational prime $p\notin \Nm_{L/\bQ}(S)$ such that $v_p(P_{D}(j_0)) > 0$ and $\left(\frac{D}{p}\right) \neq 1$, then as in \cite{MR2414789}, there is a prime $\fp \notin S$ above $p$, such that the Jacobian of $C$ has supersingular, and hence superspecial reduction. 
    \begin{enumerate}
        \item Choose a prime $l$ satisfying the conditions:  
        \begin{itemize}
            \item $l\equiv 13 \pmod{24}$,
            \item $\left(\frac{-l}{q}\right) = 1$ for every prime $q\in \Nm_{L/\bQ}(S) \backslash \{2,3\}$, 
            \item $P_{-4l}(x)$ has a real root $j(a_i)$ in a subinterval of $(-\frac{16}{27},0)$ to be specified later.
        \end{itemize}
         Let $D = -4l$. By \cite{MR2414789}*{Lemma 6} and \Cref{lem:red23}, \begin{align*}
            P_{-4l}(x) &\equiv x^{h'} \pmod{4}, \\
            P_{-4l}(x) &\equiv \pm x^{h'} \pmod{3}, \\
            P_{-4l}(x) &\equiv (27x+16)S(x)^2 \pmod{l} \text{ for some }S(x) \in\bZ[x].
        \end{align*}
        Since $v_{\mathfrak{p}}(j_0) \leq 0$ for all $\mathfrak{p}$ above $2,3$, we can choose $d$ such that
        $d\prod_{\sigma\in T} \sigma(j_0)$ is integral  for any $T\subset \Hom(L, \overline{\bQ})$ and any prime dividing $d$ lies in $\Nm_{L/\bQ}(S)$, then
        $N, dQ\in \bN$ and $(N, 6) = 1,\, (dQ,2) = 1$. 
        Suppose $l\nmid N$, then \begin{align*}
        \left(\frac{-4l}{N}\right) &= \left(\frac{-1}{N}\right)\left(\frac{l}{N}\right)= \left(\frac{-1}{N}\right)\left(\frac{N}{l}\right) \\
        &= \left(\frac{-1}{N}\right)\left(\frac{dQ}{l}\right) = \left(\frac{-1}{N}\right)\left(\frac{l}{dQ}\right)\\
        &= \left(\frac{-1}{N}\right)\left(\frac{-1}{dQ}\right) \left(\frac{-l}{dQ}\right) \\ &= \left(\frac{-1}{s(d\Nm_{L/\bQ} j_0)^{h'}}\right)\left(\frac{-1}{s'3^{[L:\bQ]}d\Nm_{L/\bQ}(j_0)}\right)\left(\frac{-l}{3^{v_3(dQ)}}\right) \\
        &= ss'(-1)^{[L:\bQ]+v_3(dQ)}.\end{align*} 
        If $[L:\bQ]+v_3(dQ)$ is odd, let \[j(a_1)\in\left(-\frac{16}{27}, -\frac{16}{27} + \min_{1\leq i \leq r} \left|j_i + \frac{16}{27}\right|\right)\]
        so that $(27j_i + 16) P_{-4l}(j_i) > 0$ for each $1\leq i\leq r$. 
        If $[L:\bQ]+v_3(dQ)$ is even, by assumption let $j_t$ be the minimal real conjugate of $j_0$ in $( -\frac{16}{27} , 0)$, and \[j(a_1)\in\left(j_t, j_{t+1}\right)\] so that $(27j_i + 16) P_{-4l}(j_i) > 0$ for each $1\leq i\leq r, i\neq t$ and $(27j_t + 16) P_{-4l}(j_t) < 0$. In either case we have \[\left(\frac{-4l}{N}\right) = -1\] and there is a prime divisor $p$ of $N$ such that  \[\left(\frac{-4l}{p}\right) = -1.\] The conditions on $l$ and $(N,6) = 1$ imply that $p\notin \Nm_{L/\bQ}(S)$ and $v_p(\Nm_{L/\bQ}(P_D(j_0))) = v_p(N) > 0$. 
        If $l | N$, then since $l\notin \Nm_{L/\bQ}(S)$ by construction, we can choose $p = l$. 
        \item Choose a prime $l$ satifying the conditions:
        \begin{itemize}
            \item $l\equiv 19\pmod{24}$,
            \item $\left(\frac{q}{l}\right)=\left(\frac{-l}{q}\right) = 1$ for every prime $q\in \Nm_{L/\bQ}(S) \backslash \{2,3\}$, 
            \item $P_{-l}(x)$ has a real root $j(a_1)$ in a subinterval of $(0, \infty)$ to be specified later.
        \end{itemize} Let $D = -l$. By \Cref{lem:red23}, \Cref{lem:pair}, and \cite{MR2414789}*{Lemma 4},
        \begin{align*}
        P_{-l}(x) &\equiv 1 \pmod{2}, \\
        P_{-l}(x) &\equiv \pm x^{h'} \pmod{3}, \\
        P_{-l}(x) &\equiv 3(27x+16)S(x)^2 \pmod{l} \text{ for some }S(x) \in\bZ[x].
        \end{align*}
    Since  $v_{\mathfrak{p}}(j_0) \geq 0$ for $\mathfrak{p}$ above $2$ and $v_{\mathfrak{p}}(j_0) \leq 0$ for $\mathfrak{p}$ above $3$, we can choose $d$ such that
        $d\prod_{\sigma\in T} \sigma(j_0)$ is integral  for any $T\subset \Hom(L, \overline{\bQ})$ and any prime dividing $d$ lies in $\Nm_{L/\bQ}(S)\backslash\{2\}$, then
        $N, dQ\in \bN$ and $(N, 6) = 1$ . 
    Suppose $l\nmid N$, then
    \begin{align*}
        \left(\frac{-l}{N}\right) &= \left(\frac{N}{l}\right) \\
        &= \left(\frac{s(d\Nm_{L/\bQ}(3(27j_0+16)))}{l}\right)\\
        &= \left(\frac{ss'3^{[L:\bQ]}dQ}{l}\right) \\ 
        &= ss'(-1)^{[L:\bQ]+v_3(dQ)+v_2(Q)}
    \end{align*}
    Let $j_t$ be the minimal real conjugate of $j_0$ in $(0, \infty)$ and $n$ ($0\leq n < r$) be the number of real conjugates of $j_0$ in $(-\frac{16}{27}, 0)$. If $[L:\bQ]+v_3(dQ)+v_2(Q)$ is odd, let \[\twocase{j(a_1)\in}{(j_t, j_{t+1})}{$n$ is odd}{(0, j_t)}{$n$ is even}\] so that $\Nm_{L/\bQ}(P_{-l}(j_0))\Nm_{L/\bQ}(27j_0+16)>0$. If $[L:\bQ]+v_3(dQ)+v_2(Q)$ is even, let \[\twocase{j(a_1)\in}{(j_t, j_{t+1})}{$n$ is even}{(0, j_t)}{$n$ is odd}\] so that $\Nm_{L/\bQ}(P_{-l}(j_0))\Nm_{L/\bQ}(27j_0+16)<0$. In either case we have \[\left(\frac{-l}{N}\right) = -1\] and there is a prime divisor $p$ of $N$ such that  \[\left(\frac{-l}{p}\right) = -1.\] The conditions on $l$ and $(N,6) = 1$ imply that $p\notin \Nm_{L/\bQ}(S)$ and $v_p(\Nm_{L/\bQ}(P_D(j_0))) = v_p(N) > 0$. 
        If $l | N$, then since $l\notin \Nm_{L/\bQ}(S)$ by construction, we can choose $p = l$. 
        
    \item Choose a prime $l$ satisfying the conditions:\begin{itemize}
        \item $l\equiv 1\pmod {24}$, 
        \item $\left(\frac{-3l}{q}\right) = 1$ for every prime $q\in \Nm_{L/\bQ}(S) \backslash \{2,3\}$, 
        \item the number of real conjugates of $j_0$ between the two real roots of $P_{-3l}(x)$ is odd, so that \[\Nm_{L/\bQ}(P_{-3l}(j_0))<0.\]
    \end{itemize} 
    Let $D = -3l$. By \cite{MR2414789}*{Lemma 4}, the only possible odd prime power in $b_{h'}$ is $3$ where $3$ is a square modulo $l$. Then by \Cref{lem:red23} and \Cref{lem:pair}, 
    \begin{align*}
        P_{-3l}(x) &\equiv 1 \pmod{2}, \\
        P_{-3l}(x) &\equiv \pm 1 \pmod{3}, \\
        P_{-3l}(x) &\equiv S(x)^2 \pmod{l} \text{ for some }S(x) \in\bZ[x].
        \end{align*}
    Since $v_{\fp}(j_0) \geq 0$ for $\fp$ above $2,3$, we can choose $d$ such that $d\prod_{\sigma\in T} \sigma(j_0)$ is integral  for any $T\subset \Hom(L, \overline{\bQ})$ and any prime dividing $d$ lies in $\Nm_{L/\bQ}(S)\backslash\{2,3\}$, then
        $N\in \bN$. Since $P_{-3l}(x)$ is constant modulo $6$ and $[L:\bQ]$ is even, we have \[d^{h'}\Nm_{L/\bQ}(P_{-3l}(j_0))\equiv 1\pmod{6}.\]
        Suppose $l\nmid P$, then \[\left(\frac{-3l}{N}\right) = \left(\frac{N}{3l}\right) = \left(\frac{-1}{3l}\right) \left(\frac{d^{h'}\Nm_{L/\bQ}(P_{-3l}(j_0))}{3l}\right)= \left(\frac{-1}{3l}\right) = -1\] since $P_{-3l}(x)$ is square modulo $l$. 
        There is a prime divisor $p$ of $N$ such that  \[\left(\frac{-3l}{p}\right) = -1.\] The conditions on $l$ and $(N,6) = 1$ imply that $p\notin \Nm_{L/\bQ}(S)$ and $v_p(\Nm_{L/\bQ}(P_D(j_0))) = v_p(N) > 0$. 
        If $l | N$, then since $l\notin \Nm_{L/\bQ}(S)$ by construction, we can choose $p = l$. 
    \end{enumerate}

\begin{rem}    \label{rem:conclude} One can try to further weaken the local conditions by considering Heegner cycles with different forms of discriminant $D$. 
For primes $p$ dividing $D$, the unpaired roots of $P_D(x)$ modulo $p$ can be predicted by checking whether a maximal order in a definite quaternion algebra ramified at $2,3,p$ contains two anticommuting CM orders of some given discriminants, as computed in \cite{MR2704678}*{3.4.2}. One can then impose conditions on the number of prime divisors of $D$ and the congruence class modulo $24$ of primes dividing $D$, so that $P_D(x)$ has no unpaired roots or $P_D(x)$ has a single unpaired root from the same elliptic point modulo each $p\nmid D$. In addition, one imposes a congruence condition on $D$ so that $\cP_{D, \cE}$ avoids intersection with one of $\cP_{-3, \cE}$ and $\cP_{-4, \cE}$ by \Cref{lem:avoidintersection}. 
The local conditions on $j_0$ are determined so that it avoids intersection with the these $\cP_D$ at $p = 2,3$, and some real conjugate of $j_0$ and some real root of  $P_D$ lie in the subinterval  corresponding to one geodesic. For $D$ not of the form $-Np$ with $N = 1, 3, 4, 8, 12, 24$, it is possible that $P_D(x)$ has multiple real roots in each subinterval, but we still expect an equidistribution result. 
\end{rem}
\bibliographystyle{amsalpha}
\bibliography{references}

\end{document}